\documentclass{amsart}
\usepackage{hyperref}
\hypersetup{
    colorlinks=true,       
    linkcolor=blue,          
    citecolor=blue,        
    filecolor=blue,      
    urlcolor=blue           
}
\usepackage[all]{xy}
\usepackage{tikz,graphicx}
\usepackage{amssymb}
\usepackage{booktabs}

\CompileMatrices

\newtheorem{theorem}{Theorem}[section]
\newtheorem{lemma}[theorem]{Lemma}

\newtheorem{proposition}[theorem]{Proposition}
\newtheorem{corollary}[theorem]{Corollary}
\theoremstyle{definition}
\newtheorem{definition}[theorem]{Definition}
\newtheorem{example}[theorem]{Example}

\newtheorem{remark}[theorem]{Remark}

\def\cc{{\mathbb C}}

\def\rr{{\mathbb R}}

\def\qq{{\mathbb Q}}

\def\Aut{\operatorname{Aut}}

\def\PGL{\operatorname{PGL}}

\usepackage{enumerate}

\begin{document}

\title[Fields of moduli  of odd signature curves]{Fields of moduli and fields of definition\\ of odd signature curves}

\author{Michela Artebani}
\address{
Departamento de Matem\'atica, \newline
Universidad de Concepci\'on, \newline
Casilla 160-C,
Concepci\'on, Chile}
\email{martebani@udec.cl}

\author{Sa\'ul Quispe}
\address{
Departamento de Matem\'atica, \newline
Universidad de Concepci\'on, \newline
Casilla 160-C,
Concepci\'on, Chile}
\email{squispe@udec.cl}

\thanks{The authors have been partially
supported by Proyecto FONDECYT Regular 1110249.}

\maketitle

\begin{abstract}
Let $X$ be a smooth projective  curve of genus $g\geq2$ defined over a field $K$. We show that $X$ can be defined over its field of moduli $K_X$ if the signature of the covering $X\to X/\Aut(X)$ is of type $(0;c_1,\dots,c_k)$, where some $c_i$ appears an odd number of times.
This result is applied to $q$-gonal curves and to plane quartics. 
 \end{abstract}

\section*{Introduction}
 Let $X$ be a smooth projective curve of genus $g$ defined over a field $K$ and let $K_X$ be its field of moduli  (see Section 1, Definition \ref{def1}). It is well known that $X$ can be defined over $K_X$ if either $g=0,1$ or the automorphism group  of $X$ is trivial.
 However, there are examples of curves which can not be defined over $K_X$, as first observed by Earle and Shimura in \cite{Earle, Shi}.
In \cite{hu2} B. Huggins studied this problem for hyperelliptic curves in characteristic $p\neq 2$, proving that 
a hyperelliptic curve $X$ of genus $g \geq 2$ with hyperelliptic involution $\iota$ can be defined over $K_X$ provided that $\Aut(X)/\langle\iota\rangle$ is not cyclic or is cyclic of order divisible by $p$. 

The first examples of non-hyperelliptic curves not definable over their field of moduli have been given in  \cite{hu2} and \cite{H}.

Recently R. Hidalgo \cite{hi} considered complex curves $X$ such that the natural covering $\pi_X:X\rightarrow X/\Aut(X)$ has signature of the form $(0; a, b, c, d)$, proving that $X$ can be defined over its field of moduli if $d\notin \{a, b, c\}$.
In this paper we observe that such result can be extended to {\em odd signature curves}, i.e. curves such that the signature of $\pi_X$ is of the form $(0; c_i,\cdots,c_r)$ where some $c_i$ appears exactly an odd number of times. More precisely, we prove the following result, which is a consequence of \cite[Theorem 3.1]{De}.

\begin{theorem}\label{sa}
 Let $X$ be a smooth projective  curve  of genus $g\geq2$ defined over a field $K$. If $X$ is an odd signature curve, then $K_X$ is a field of definition for $X$.
\end{theorem}
This result implies that non-normal $q$-gonal curves can be defined over their field of moduli and 
that plane quartics can be defined over their field of moduli if $|\Aut(X)|>4$.
In the last section of the paper we construct examples of plane quartics with 
$\Aut(X)\cong C_2$ which can not be defined over their field of moduli and we prove that,
in case $\Aut(X)\cong C_2\times C_2$,  the field of moduli relative to the extension $\cc/\rr$ is always a field of definition. This implies the following.

\begin{theorem}\label{qua}
Let $X$ be a smooth plane quartic over $\cc$ which is isomorphic to its conjugate. 
If  $\Aut(X)$ is not cyclic of order two, then $X$ can be defined over $\rr$.
 \end{theorem}

\subsubsection*{Acknowledgments}
This paper is part of the PhD thesis of the second author.
We are grateful to Ruben Hidalgo and to Antonio Laface for several useful comments.

\section{Preliminaries}
 Let $X$ be a smooth projective  curve defined over a field $K$.  A subfield $N$ of $K$ is a \emph{field of definition} of $X$ if there exists a curve $X'$ defined over $N$ such that $X'$ is isomorphic to $X$ over $K$.
 Moreover, we say that $X$ is \emph{definable} over $N$ if there exists a curve $X'$ defined over $N$ such that $X'$ is isomorphic to $X$ over $\bar K$.
\begin{definition}\label{def1}
 Let $K$ be a field, $\overline{K}$ be an algebraic closure of $K$ and $X$ be a curve defined over $K$.   The \emph{field of moduli} $K_X$ of $X$ is the intersection of all fields of definition of $X$, seen as a curve over $\overline{K}$.
\end{definition}

Another definition for the field of moduli, relative to a given field extension $F/L$, is given as  follows.
If $P\in F[x_0,\cdots,x_n]$ and $\sigma\in \Aut(F/L)$, then $P^{\sigma}$ denotes the  polynomial obtained by applying $\sigma$ to the coefficients of $P$. If a curve $X$ is defined as the zero locus of the homogeneous polynomials $P_1,\cdots,P_s\in F[x_0,\cdots,x_n]$, then the polynomials $P_1^{\sigma},\cdots,P_s^{\sigma}$ define a new smooth projective curve $X^{\sigma}$.
\begin{definition}
The \emph{field of moduli} of $X$ \emph{relative to the extension} $F/L$, denoted by $M_{F/L}(X)$, is the fixed field of the group $$U_{F/L}(X):=\{\sigma\in \Aut(F/L):\ X\  \textrm{is isomorphic to }X^{\sigma}\ \textrm{over } F\}.$$
\end{definition}

Let $P$ be the prime field of $K$. By a theorem of Koizumi (see \cite{Koi} and \cite[Theorem 1.5.8]{hu1}) the field of moduli $M_{\bar K/P}(X)$ is a purely inseparable extension of the field of moduli $K_X$. In particular these two fields coincide if $K$ is a perfect field. For example, if $K=\mathbb C$,  then $K_X=M_{\mathbb C/\mathbb Q}(X)$.
The relationship between $K_X$ and the fields of moduli of $X$ relative to Galois extensions  is given by the following result (see \cite[Theorem 1.6.8]{hu1}).

\begin{theorem}\label{sqa}
Let $X$ be a smooth projective algebraic curve defined over a field $K$ and $K_X$ be the field of moduli of $X$. Then $X$ is definable over $K_X$ if and only if given any algebraically closed field $F\supseteq K$, and any subfield $L\subseteq F$ with $F/L$ Galois, $X$ (seen as a curve over $F$) can be defined over the field $M_{F/L}(X)$.
\end{theorem}

Given a  smooth projective algebraic curve $Y$ defined over $L$, a branched covering  $\phi : X \rightarrow Y$ defined over $F$ and $\sigma\in\Aut(F/L)$, we  denote by $\phi^{\sigma}: X^{\sigma}\rightarrow Y^{\sigma}$ the branched covering obtained by applying $\sigma$ to the defining polynomials of $\phi$.

Assume now that a curve $L$ is a field of definition of a curve $X$ over $F$, i.e. there exists an isomorphism $g:X\to Y$, where $Y$ is a curve defined over $L$. If $\sigma\in \Aut(F/L)$, then $f_{\sigma}:=(g^{\sigma})^{-1}\circ g:X\to X^{\sigma}$ is an isomorphism (observe that $Y=Y^{\sigma}$)  and $f_{\tau\sigma}=f_{\sigma}^{\tau}\circ f_{\tau}$ holds for all $\sigma,\tau\in \Aut(F/L)$.
The following theorem by A. Weil shows that the latter condition is also sufficient  for the field $L$ to be a field of definition for $X$.

\begin{theorem}[Weil \cite{We}]\label{Wa}
Let $X$ be a smooth projective algebraic curve defined over a field $F$ and let $F/L$ be a Galois extension. 
If for every $\sigma\in \Aut(F/L)$ there is an isomorphism $f_{\sigma}: X\rightarrow X^{\sigma}$ defined over $F$ such that  the compatibility
condition $f_{\tau\sigma} = f_{\sigma}^{\tau}\circ f_{\tau} $ holds for all $\sigma, \tau\in \Aut(F/L)$, 
then there exist a smooth projective algebraic curve $Y$ defined
over $L$ and an isomorphism $g: X\rightarrow Y$ defined over $F$ such that $g^{\sigma}\circ f_{\sigma} = g$.
\end{theorem}

The following result by D\`ebes-Emsalem  is a consequence of Weil's theorem and provides a sufficient condition for the curve $X$ to be defined over the field $M_{F/L}(X)$ (see \cite[\S 2.4]{DD} for the definition of field of moduli of a covering).

\begin{theorem}[D\`ebes-Emsalem \cite{De}]\label{de}
Let $F/L$ be a Galois extension and $X$ be a smooth projective curve of genus $g\geq 2$ defined over $F$ with $L:=M_{F/L}(X)$. Then there exist a smooth projective  curve $B$ defined over $L$ and  a Galois branched covering $\phi: X\rightarrow B$ defined over $F$, with $\Aut(X)$ as its deck group, so that $M_{F/L}(\phi) = L$. 
Moreover, if $B$ contains at least one $L$-rational point outside of the branch locus of $\phi$, then $L$ is also a field of definition of $X$. 
\end{theorem}


 \begin{remark}\label{de1} The condition $L:=M_{F/L}(X)$ in Theorem \ref{de} is not restrictive since by \cite[Proposition 2.1]{De} the field of moduli relative to the extension $F/M_{F/L}(X)$ is $M_{F/L}(X)$.
\end{remark}

\section{Proof of the theorem}
Let $\phi: X\rightarrow X/G$ be a branched Galois covering between smooth projective curves and let $q_1,\cdots,q_r$ be its branch points. 
 The \emph{signature} of $\phi$ is defined as $(g_0; c_1,\cdots,c_r)$, where $g_0$ is the genus of  $X/G$ and $c_i$ is the ramification index of any point in $\phi^{-1}(q_i)$.
The \emph{branch divisor} of $\phi$, denoted by $D(\phi)$, is the divisor of $X/G$ defined by
$D(\phi)=\sum_{i=1}^rc_i q_i.$

\begin{definition}
 A smooth projective  curve $X$ of genus $g\geq2$ has \emph{odd signature} if the signature of the covering $\pi_X:X\rightarrow X/\Aut(X)$ is of the form $(0; c_1,\cdots, c_r)$ where some $c_i$ appears exactly an odd number of times.

\end{definition}

\begin{definition}
Let $B$ be a smooth projective curve defined over a field $L$. A divisor $D=p_1+ \cdots + p_r$ of $B$ is called $L$-\emph{rational} if for each $\sigma\in \Aut(\bar L/L)$ we have that $D^{\sigma}:=\sigma(p_1)+\cdots +\sigma(p_r)=D$.
\end{definition}
The following is an easy consequence of Riemann-Roch theorem and the fact that a curve of genus zero with a $L$-rational point is isomorphic to $\mathbb{P}^1(L)$ (see also  \cite[Lemma 4.0.4.]{hu1}).
\begin{lemma}\label{dkr}
Let $B$ be a smooth projective curve of genus $0$ defined over an infinite field $L$ and suppose that $B$ has an $L$-rational divisor $D$ of odd degree. Then $B$ has infinitely many $L$-rational points.
\end{lemma}

 \begin{lemma}\label{branch}
Given a Galois branched covering $\phi:X\to X/G$ as before defined over $F$, we have $D(\phi^{\sigma})=D(\phi)^{\sigma}$ for any $\sigma\in \Aut(F/L)$.
\end{lemma}

\begin{proof}
Observe that  $\sigma\circ \phi=\phi^{\sigma}\circ \sigma$, where we denote by $\sigma$ the bijection acting as $\sigma$ on the coordinates of the points of $X$ and $X/G$. 
Thus $q_i$ belongs to the support of $D(\phi)$ if and only if  $\sigma(q_i)$ is in the support of $D(\phi^{\sigma})$ and the fibers over the two points have the same cardinality.
 \end{proof}

The proof of Theorem \ref{sa} follows from Theorem \ref{sqa} and the following result.
\begin{theorem}\label{sa2}
 Let $X$ be a smooth projective  curve  of genus $g\geq2$ defined over an algebraically closed field $F$ and let $L\subset F$ be a subfield such that $F/L$ is Galois. If $X$ is an odd signature curve, then $M_{F/L}(X)$ is a field of definition for $X$.
\end{theorem}

\begin{proof}
By Remark \ref{de1} we can assume that $M_{F/L}(X)=L$.
By Theorem \ref{de} there exists a canonical $L$-model $B$ of $X/\Aut(X)$ and a commutative diagram:
\[
\xymatrix@1{
X\ar[d]_-{\pi_X}\ar[rr]^-{f_{\sigma}} & & X^{\sigma}\ar[d]^-{\pi_X^{\sigma}}\\
X/\Aut(X)\ar[rr]^-{h_{\sigma}}\ar[rd]_-{g} & & (X/\Aut(X))^{\sigma}\ar[ld]^-{g^{\sigma}}\\
& B &
}
\]
where $\sigma\in \Aut(F/L)$ (this coincides with $U_{F/L}$ by \cite[Proposition 2.1]{De}) and $f_{\sigma}, h_{\sigma}, g$  are isomorphisms.
Let $\phi=g\circ \pi_X$. The fact that $f_{\sigma}$ is an isomorphism and Lemma \ref{branch} imply that $D(\phi)=D(\phi^{\sigma})=D(\phi)^{\sigma}$, i.e. $D(\phi)$ is an $L$-rational divisor. Also, as $g$ is an isomorphism,  $D(\phi)=g(D(\pi_X))$ and  $\phi$ has the same signature of $\pi_X$.  If $q_{1}, \cdots, q_{2k+1}$ are the points in the support of $D(\phi)$ with the same coefficient $c_i$, then the divisor $q_1+\cdots+q_{2k+1}$ is an $L$-rational divisor of odd degree.

If $L$ is infinite this implies, by Lemma \ref{dkr}, that $B$ has an $L$-rational point outside of the branch locus of $\phi$, thus $X$ can be defined over $L$ by Theorem \ref{de}.
In case $L$ is finite the result follows from \cite[Corollary 2.11]{hu2}. \end{proof}

\section{Cyclic $q$-gonal curves}

Let $F$ be an algebraically closed field of characteristic $p\not=2$ and let $X$ be an algebraic curve of genus $g\geq2$ defined over $F$.
 If the automorphism group  of  $X$  contains a cyclic subgroup $C_q$, where $q$ is a prime number, such that  $X/C_q$ has genus zero, then the curve is called a  \emph{cyclic} $q$-\emph{gonal curve}.
 If in addition $C_q$ is normal in  $\Aut(X)$, then $X$ is called a \emph{normal cyclic} $q$-\emph{gonal curve}. In this case the \emph{reduced automorphism group}  $\overline{\Aut(X)} := \Aut(X)/C_q$ 
 is isomorphic to a finite subgroup of  $\PGL_2(F)$. 
  
 In case $\overline{\Aut(X)}$ is not cyclic B. Huggins \cite[Theorem 5.3]{hu2} and A. Kontogeorgis \cite[Proposition 3.2]{Ko} proved the following theorem.
  \begin{theorem}\label{Gu1}
Let $K$ be a perfect field of characteristic $p\not=2$ and let $F$ be an algebraic closure of $K$. Let $X$ be a normal cyclic $q$-gonal curve over $F$ such that $\overline{\Aut(X)}$ is not cyclic or that $\overline{\Aut(X)}$ is cyclic of order divisible by $p$. Then $X$ can be defined over its field of moduli relative to the extension $F/K$.
\end{theorem}

In case $\overline{\Aut(X)}$ is cyclic of order $n$ and $p=0$, then $X$ is isomorphic to a curve with equation $y^q=f(x)$, where $f$ is as given in Table \ref{table1}.
Observe that $\overline{\Aut(X)}$ is generated by $\nu(x)=\zeta_nx$, where $\zeta_n$ is a primitive $n$-th root of unity.
The three cases in Table \ref{table1}  differ by the number $N$ of branch points of the cover $X\to X/C_q$ fixed by $\nu$.

\begin{corollary}\label{h1}
Let $X$ be a normal cyclic $q$-gonal curve of genus $g\geq 2$ defined over a field $K$ of characteristic zero such 
that $\overline{\Aut(X)}$ is cyclic of order $n\geq 2$ and let  $N$ be as above.
If either $N=1$, or $N=0$ and $\frac{2g-2+2q}{n(q-1)}$ is odd, or $N=2$ and $\frac{2g}{n(q-1)}$ is odd,
then $X$ is definable over $K_X$.
\end{corollary}
 
\begin{proof}
 The signature of  the covering $\pi_X:X\to X/\Aut(X)$  is given in Table \ref{qgonal}.
 If $N=1$ then clearly $X$ has odd signature. Otherwise, if $N=0$, the number of branch points 
 with ramification index $q$ equals  $\frac{2g-2+2q}{n(q-1)}$ by the Riemann-Hurwitz formula,
 thus again $X$ has odd signature. Similarly for $N=2$.
 Thus the result follows from  Theorem \ref{sa}.
\end{proof}

\begin{table}[h] \label{table1}
$$\begin{tabular}{|c|c|c|c|c|c|}
\hline
$N$   &  signature of $\pi_X$ &  $f(x)$ \\
\hline
0  &$(0;n, n, q,\ldots,q)$ & $x^{nt}+\cdots +a_ix^{n(t-i)}+\cdots + a_{t-1}x^n+1$\\[2pt]
&     & where $q|nt$\\
\hline
  1  &$(0;n, nq, q,\ldots,q)$ & $x^{nt}+\cdots +a_ix^{n(t-i)}+\cdots + a_{t-1}x^n+1$\\[2pt]
 &     & where $q\not|nt$\\
\hline
 2   &$(0;nq, nq,  q,\ldots,q)$ & $x(x^{nt}+\cdots +a_ix^{n(t-i)}+\cdots + a_{t-1}x^n+1)$\\[2pt]
 &   & where $q\not| nt+1$\\
\hline
\end{tabular}$$
\vspace{0.2cm}

\caption{Cyclic $q$-gonal curves with $\overline{\Aut(X)}=C_n$} \label{qgonal}
\end{table}
\vspace{-0.3cm}
 
We will now construct examples of cyclic $q$-gonal curves not definable over their field of moduli  following \cite{hu1, hu2}. Let $m,n>1$ be two integers,
$a_1,\dots,a_m\in \mathbb{C}$ and consider the polynomial  
\begin{equation}\label{Ec2}
f(x):=\prod_{1\leq i \leq m}(x^n-a_i)(x^n+1/\bar a_i).
\end{equation}

We will look for such an $f$ with the following properties: $|a_i|\neq|a_j|$  if $i\neq j$, 
$a_i/\bar a_i\neq a_j/\bar a_j$ if $i\neq j$, $|a_i|\neq |1/a_j|$ for all $i,j$, $f(0)=-1$.
Moreover, if $n=3$ we ask that  the following automorphism does not map the zero set of $f$ into itself:
 $$\tau:\ x\mapsto \frac{-(x-\sqrt{3}-1)}{x(\sqrt{3}-1)+1}.$$
 We observe that such polynomials exist for any $m,n$:  
 for $n\not=3$ we can consider
 $$f(x)=\prod_{1\leq l\leq m}(x^n-(l+1)\kappa ^l)(x^n+\frac{\kappa^l}{l+1}),$$
and for $n=3$ the polynomial:
 $$f(x)=(x^3-\alpha^3)(x^3+\frac{1}{\alpha^3})\prod_{1\leq l\leq m-1}(x^3-(l+1)\kappa ^l)(x^3+\frac{\kappa^l}{l+1}),$$
  where $\kappa$ is a primitive $m$-th root of $(-1)^{m-1}$ and $\alpha=-(2+\sqrt 3)$ (observe that  $\tau(\alpha)=\alpha$).

 \begin{lemma}\label{Ha1}
Let $X$ be a cyclic $q$-gonal curve over $\mathbb{C}$ given by $y^q=f(x)$, where $f$ is as in (\ref{Ec2}) and satisfies the properties mentioned above. Then:
\begin{enumerate}[i)]
\item $\Aut(X)$ is generated by $\iota(x, y)=(x, \zeta_q y)$  and $\nu(x, y)=(\zeta_n x, y)$;
\item the signature of $\pi_X$ is $(0; q,\ldots,q, n, n)$  if $q|2mn$ and $(0; q,\ldots,q, n, qn)$ otherwise, where $q$ appears $2m$-times. 
\end{enumerate}
\end{lemma}

\begin{proof}
Observe that $ii)$ is obvious by Table \ref{qgonal}.
 If $n\neq3$, then $i)$ follows from \cite[Lemma 6.1]{hu2} and its proof (which does not depend on the fact that $m$ is odd). For $n=3$ we need to exclude the missing case $\langle\overline{\nu}\rangle<\overline\Aut(X)\cong A_4$, where $\overline{\nu}$ is the image of $\nu$ in $\overline\Aut(X)$. Suppose we are in this case, then by \cite[Corollary 3.2]{bt} $\tau$ would be an automorphism of $f(x)$, giving a contradiction.
\end{proof}

 \begin{table}
$$\begin{tabular}{|c|c|c|c|c|}
\hline
$q$ & signature of $\pi_X$  & $g$ & $\Aut(X)$\\
\hline
$3$ & $(0; 2, 3, 8)$ & $2$ & ${\rm GL}(2, 3)$\\
\hline
$3$ & $(0; 2, 3, 12)$  & $3$ & ${\rm SL}(2, 3)/{\rm CD}$\\
\hline
$5$ & $(0; 2, 4, 5)$  & $4$ & $S_5$\\
\hline
$7$ & $(0; 2, 3, 7)$  & $3$ & ${\rm PSL}(2, 7)$\\
\hline
$q\geq5$ & $(0; 2, 3, 2q)$  & $\frac{(q-1)(q-2)}{2}$ & $(C_q\times C_q)\rtimes S_3$\\
\hline
$q\geq3$ & $(0; 2, 2, 2, q)$  & $(q-1)^2$ & $(C_q\times C_q)\rtimes V_4$\\
\hline
$q\geq3$ & $(0; 2, 4, 2q)$  & $(q-1)^2$ & $(C_q\times C_q)\rtimes D_4$\\
\hline
\end{tabular}$$
\caption{Non-normal $q$-gonal curves.}\label{nonnormal}
\end{table}

The following generalizes  \cite[Proposition 5.0.5]{hu1} and \cite[Proposition 6.2]{hu2}. Observe that if $q$ does not divide $mn$, then $X$ is an odd signature curve by the previous Lemma, thus it can be defined over its field of moduli 
relative to the extension $\cc/\rr$.

\begin{proposition}
Let $X$ be a  cyclic $q$-gonal curve over $\mathbb{C}$ given by $y^q=f(x)$, where $q>2$, $f$ is as in (\ref{Ec2}) and satisfies the properties mentioned above, $m,n>1$ and $q|mn$. The field of moduli of $X$ relative to the extension $\mathbb{C}/\mathbb{R}$ is $\mathbb{R}$ and is a field of definition of $X$ if and only if $n$ is odd. 
\end{proposition}

\begin{proof}
Observe that $X$ is isomorphic to the conjugate curve  $$\bar X:\ y^q=\prod_{1\leq i \leq m}(x^n-\bar a_i)(x^n+1/a_i)$$  
by the isomorphism $$\mu(x, y)=\left(\frac{1}{\zeta_{2n} x}, \frac{\zeta_{2q} y}{x^{2mn/q}}\right).$$ 
By Lemma \ref{Ha1} the automorphism group of $X$ is generated by $\iota$ and $\nu$, thus any isomorphism between $X$ and $\bar X$ is of the form  $\mu \iota^j\nu^k$, where $0 \leq j\leq q-1$ and $0\leq k\leq n-1$.  
An easy computation shows that
 \begin{eqnarray*}
\overline{(\mu\nu^k)}\mu\nu^k & = 
(\tau')^{2k+1}\nu^{2k+1},
\end{eqnarray*}
where $\tau'(x, y)=(x, \zeta_n^{mn/q}y)$ .
 Moreover, since $\iota$ commutes with $\mu$ and $\nu$: 
 $$\overline{(\mu\iota^j\nu^k)}\mu\iota^j\nu^k=\bar\mu \iota^{-j}\overline{\nu^{k}}\mu\iota^j\nu^k=\bar\mu \overline{\nu^{k}}\mu\nu^{k}=\overline{(\mu\nu^k)} \mu\nu^k.$$

In case $n$ is even the cocycle condition in Theorem \ref{Wa} does not hold since  $\nu^{2k+1}\not=id$ for any $k$, thus $X$ cannot be defined over $\mathbb{R}$.  
Otherwise, if $n$ is odd, we have $\overline{(\mu\nu^k)}\mu\nu^k=id$ with $k=(n-1)/2$,
so that $X$ can be defined over $\rr$.
\end{proof}

\begin{corollary}\label{h2}
Let $X$ be a non-normal $q$-gonal curve defined over a field $K$ of characteristic zero. Then $X$ is definable over $K_X$.
\end{corollary}
\begin{proof}
By \cite[Theorem 8.1]{Wo} the signature of $\pi_X$ is given in Table \ref{nonnormal}.
In any case $X$ has odd signature, thus the result follows from Theorem \ref{sa}.
\end{proof}

\section{Plane quartics}
In this section $X$ will be a smooth plane quartic defined over an algebraically closed field of characteristic zero. Table \ref{quartics} lists all possible automorphism groups of smooth plane quartics. Moreover, for each group, it gives the equation of a plane quartic having this group as automorphism group (n.a. means ``not above'', i.e. not isomorphic to other models above it in the table) and the signature of  the covering $\pi_X$ (see \cite[Theorem 16 and \S 2.3]{ba}). 
 \begin{table}[h]
$$\begin{tabular}{|c|c|c|}
\hline
$G$ & equation & signature\\
\hline
${\rm PSL}_2(7)$ &  $z^3y+y^3x+x^3z$ & $(0; 2, 3, 7)$\\
\hline
$S_3$ & $z^4+a z^2yx+z(y^3+x^3)+b y^2x^2$ & $(0; 2, 2, 2, 2, 3)$\\
      & $a\neq b, \ ab\neq0$ & \\
\hline
$C_2\times C_2$ & $x^4+y^4+z^4+a x^2y^2+b x^2z^2+c y^2z^2$ & $(0; 2, 2, 2, 2, 2, 2)$\\
 & $a\neq b,\ a\neq c,\ b\neq c$& \\
\hline
$D_4$ & $x^4+y^4+z^4+a z^2(y^2+x^2)+b y^2x^2$ & $(0; 2, 2, 2, 2, 2)$\\
  & $a\neq b,\ a\neq0 $ & \\
\hline
$S_4$ & $x^4+y^4+z^4+a(z^2y^2+z^2x^2+y^2x^2)$ & $(0; 2, 2, 2, 3)$\\
 & $a\neq0,\ \frac{-1\pm\sqrt{-7}}{2}$ & \\
\hline
$C_4^2\rtimes S_3$ & $z^4+y^4+x^4$ & $(0; 2, 3, 8)$\\
\hline
$C_4\circledcirc (C_2)^2$ & $z^4+y^4+x^4+a z^2y^2$ & $(0; 2, 2, 2, 4)$\\
 & $a\neq0, \pm2, \pm6, \pm(2\sqrt{-3})$ & \\
\hline
$C_4\circledcirc A_4$ & $x^4+y^4+xz^3$ & $(0; 2, 3, 12)$\\

\hline
$C_6$ & $z^4+a z^2y^2+y^4+yx^3$ & $(0; 2, 3, 3, 6)$\\
 & $a\neq 0$ & \\
\hline
$C_9$ & $z^4+zy^3+yx^3$ & $(0; 3, 9, 9)$\\
\hline

$C_3$ & $z^3L_1(y, x)+L_4(y, x)\ (n.a)$ & $(0; 3, 3, 3, 3, 3)$\\
\hline
$C_2$ & $z^4+z^2L_2(y, x)+L_4(y, x)\ (n.a.)$ & $(1; 2, 2, 2, 2)$\\
\hline
\end{tabular}
 $$
 \vspace{0.3cm}
 \caption{Automorphisms of plane quartics.}\label{quartics}
 \end{table}
 
  \noindent Table \ref{quartics} and Theorem \ref{sa2} imply the following result.
\begin{corollary}\label{fmqua}
Let $X$ be a smooth plane quartic defined over an algebraically closed field $K$ of characteristic zero. If either $\Aut(X)$ is trivial or $|\Aut(X)|>4$, then $X$ is definable over $K_X$.
\end{corollary}

Observe that the hypothesis in the Corollary is equivalent to ask that $\Aut(X)$ is not isomorphic to either $C_2$ or $C_2\times C_2$.  
We will now construct a plane quartic $X$ with $\Aut(X)\cong C_2$ and of field of moduli $\mathbb{R}$ but not definable over $\mathbb{R}$.
Consider the family $X_{a_1,a_2,a_3}$ of plane quartics defined  by
$$y^4+y^2(x-a_1z)(x+\frac{1}{a_1}z)+(x-a_2z)(x+\frac{1}{\bar a_2}z)(x-a_3z)(x+\frac{1}{\bar a_3}z)=0,$$
where $a_1\in \mathbb{R}$ and $a_2a_3\in \rr$.
The following Lemma implies that  the generic curve in the family is smooth and has automorphism group of order two.

\begin{lemma}\label{Cu1} The plane quartic $X_{a_1,a_2,a_3}$ with $a_1=1, a_2=1-i$ and $a_3=2(i-1)$ is smooth and its automorphism group is generated by  $\nu(x: y: z)=(x: -y: z).$
\end{lemma}

\begin{proof} We recall that any automorphism of a smooth plane quartic is induced by an element of ${\rm PGL}(3,\cc)$.
If $\Aut(X)$ properly contains the cyclic group generated by $\nu$, then it contains a subgroup isomorphic to either $C_2\times C_2, C_6$ or $S_3$  by \cite[pag.\,26]{ba}.
We will now exclude each of these cases.

The first case can be excluded because an explicit computation shows that there is no involution, except $\nu$, which preserves the four fixed points of $\nu$.

 Now suppose that $\Aut(X)$ contains a cyclic subgroup of order $6$ generated by $\alpha$ with $\nu=\alpha^3$.  The automorphism $\tau:=\alpha^2$ induces an order three automorphism $\overline{\tau}$ on the elliptic curve $E:=X/\langle \nu\rangle$ having fixed points. This is a contradiction  since the curve $E$ (whose equation can be obtained replacing $y^2$ with $y$ in the equation of $X$)  has $j$-invariant distinct from zero.
  
Finally, suppose that $\Aut(X)$ contains a subgroup $\langle \nu,\gamma\rangle$ isomorphic to $S_3$.
 Here we will apply a method suggested by F. Bars \cite{ba}.
 By \cite[Theorem 29]{ba}, up to a change of coordinates the equation of $X$ takes the following form:
  \[
 (u^3 + v^3) w + u^2v^2 + auvw^2 + bw^4=0.
 \]
 and the generators of $S_3$ with respect to the coordinates $(u,v,w)$ are
 \[ \alpha:=\left( \begin{array}{ccc}
0 & 1 & 0 \\
1 & 0 & 0 \\
0 & 0 & 1 \end{array} \right), \qquad \beta:=\left( \begin{array}{ccc}
\zeta_3 & 0 & 0 \\
0 & \zeta_3^2 & 0 \\
0 & 0 & 1 \end{array} \right).\]
Thus there exists $A\in \PGL(3,\cc)$ such that $A\alpha A^{-1}=\nu, A\beta A^{-1}=\gamma.$
 The first condition implies that $A$ is an invertible matrix of the following form
 \[ A=\left( \begin{array}{ccc}
a & a & c \\
d & -d & 0 \\
g & g & l \end{array} \right).
\]
Note that $X$ has exactly four bitangents $x=s_jz,\ j=1,2,3,4$ invariant under the action of the involution $\nu$,
where $s_j$  are the zeros of $$\triangle=(x^2-1)^2-4(x-(1+i))(x+\frac{1}{1-i})(x-2(-1+i))(x-\frac{1}{2(1+i)}).$$
Let $b_{j1}=(s_j,q_j,1), b_{j2}=(s_j,-q_j,1)$ be the two tangency points of the line $x=s_jz$.
On the other hand, observe that the line $w=0$ is invariant for $\alpha$ and it is bitangent to $X$ at $p_1=(1:0:0), p_2=(0:1:0)$.
Thus for some $j$ we have $\{Ap_1,Ap_2\}=\{b_{j1}, b_{j2}\}$,
 from which we get $a=s_jg$, $d=\pm q_jg$.
By means of these remarks and using the Magma  \cite{Magma} code available at this webpage
\vspace{0.2cm}

\url{https://sites.google.com/site/squispeme/home/fieldsofmoduli}
\vspace{0.2cm}

\noindent we proved that  $\gamma=A\beta A^{-1}$ is not an automorphism of $X$.
\end{proof}

\begin{proposition}
 Let $X_{a_1,a_2,a_3}$ as defined previously with $\Aut(X_{a_1,a_2,a_3})\cong C_2$. Then the field of moduli of $X_{a_1,a_2,a_3}$ relative to the extension $\mathbb{C}/\mathbb{R}$ is $\mathbb{R}$ and is not a field of definition for $X$.
\end{proposition}

\begin{proof}
Observe that the following is an isomorphism between $X:=X_{a_1,a_2,a_3}$ and its conjugate $\bar X$: 
$$\mu(x: y: z)=(-z: iy: x).$$
Since $\Aut(X)$ is generated by $\nu(x: y: z)=(x: -y: z)$, the only isomorphisms between $X$ and $\bar X$ are $\mu$ and $\mu\nu$. 
Observe that $\bar \mu\mu=\nu$ and $\overline{(\mu\nu)}\mu\nu=\nu$.
Therefore Weil's cocycle condition from Theorem \ref{Wa} does not hold, so $X$ cannot be defined over $\mathbb{R}$.
\end{proof}

 Finally we study plane quartics with automorphism group isomorphic to $C_2\times C_2$, which belong to the  following family:
 $$X_{a,b,c}:\ x^4+y^4+z^4+ax^2y^2+bx^2z^2+cy^2z^2=0,$$
 where $a,b,c\in \cc$. It can be easily checked that $X_{a,b,c}$ is smooth unless $a^2+b^2+c^2-abc=4$ or some of $a^2,b^2,c^2$ is equal to $4$.
A subgroup of ${\rm Aut}(X_{a,b,c})$ isomorphic to $C_2\times C_2$  is generated by the involutions:
 $$\iota_1(x:y:z)=(-x:y:z),\quad \iota_2(x:y:z)=(x:-y:z).$$
 We will denote by $G\cong S_3\ltimes (C_2\times C_2)$ the group acting on the triples $(a,b,c)\in \cc^3$ generated
 by
 $$g_1(a,b,c)=(b,a,c),\quad g_2(a,b,c)=(b,c,a),$$
 $$g_3(a,b,c)=(-a,-b,c),\quad g_4(a,b,c)=(a,-b,-c).$$
  The following comes from a result by E.W. Howe \cite[Proposition 2]{Ho}, observing that any isomorphism between $X_{a,b,c}$ and $X_{g(a,b,c)}$, $g\in G$, is defined over $\qq(i)$.

  \begin{proposition}\label{howe} 
  If $a^2,b^2,c^2$ are pairwise distinct, then ${\rm Aut}(X_{a,b,c})\cong C_2\times C_2$. 
Moreover, if $F$ is a field containing $\qq(i)$, then a plane quartic $X_{a',b',c'}$ is isomorphic to $X_{a,b,c}$ over $F$ if and only if $g(a,b,c)=(a',b',c')$ for some $g\in G$
 \end{proposition}

 
The following result and Corollary \ref{fmqua} prove Theorem \ref{qua}.
 \begin{corollary} Let $X_{a,b,c}$ as before with $a^2,b^2,c^2$ pairwise distinct.
If the field of moduli of $X_{a, b, c}$ relative to the extension $\mathbb{C}/\mathbb{R}$ is $\mathbb{R}$, then it is a field of definition for $X_{a, b, c}$.
\end{corollary}

\begin{proof}
By Proposition \ref{howe}, the curve $X_{a, b, c}$ and its conjugate $X_{\bar a, \bar b, \bar c}$ are isomorphic over $\cc$ if and only if $g(a, b, c)=(\bar a, \bar b, \bar c)$  for some $g\in G$.
It is enough to consider the generators of $G$.
 \begin{enumerate}[i)]
 \item If $(\bar a, \bar b, \bar c)=g_1(a, b, c)=(b, a, c)$ then $\mu:X_{a, b, c}\rightarrow  X_{b, a, c},\ \mu(x: y: z)= (x: z: y)$ is an isomorphism 
and $\bar \mu\mu=id$.  
\item If $(\bar a, \bar b, \bar c)=g_2(a, b, c)=(b, c, a)$, i.e., $\bar a=b,\ \bar b=c,\ \bar c=a,$
 then $a=b=c\in\mathbb{R}$, contradicting the hypothesis on $a,b,c$. So this case does not appear.
\item If $(\bar a, \bar b, \bar c)=g_3(a, b, c)=(-a, -b, c)$ then $\mu:X_{a, b, c}\rightarrow X_{-a, -b, c},\ \mu(x: y: z)= (ix: y: z)$  is an isomorphism
 and  $\bar\mu \mu=id$.  
\item If $(\bar a, \bar b, \bar c)=g_4(a, b, c)=(a, -b, -c)$ then $\mu:X_{a, b, c}\rightarrow X_{a, -b, -c},\ \mu(x: y: z)= (x: y: iz)$ is an isomorphism
 and  $\bar\mu \mu=id$. 
 \end{enumerate} 
Therefore by Weil's Theorem we conclude that  $X_{a, b, c}$ can be defined over $\mathbb{R}$.
\end{proof}

\noindent  We now determine the field of moduli of a plane quartic in the family. Consider the following polynomials invariant for $G$:
 $$j_1(a,b,c)=abc,\ j_2(a,b,c)=a^2+b^2+c^2,\ j_3(a,b,c)=a^4+b^4+c^4;$$

 \begin{proposition}\label{fm}
 Let $F/K$ be a general Galois extension with $\qq(i)\subset F\subset \cc$ and let $a,b,c\in F$ such that $a^2,b^2,c^2$ are pairwise distinct and  $X_{a,b,c}$ is smooth. The field of moduli of $X_{a,b,c}$ relative to the extension $F/K$ equals $K(j_1,j_2,j_3)$. 
 \end{proposition}
 \proof 
 The morphism $\varphi(a,b,c)=(abc, a^2+b^2+c^2, a^4+b^4+c^4)$ has degree $24=|G|$ and clearly $\varphi(g(a,b,c))=\varphi(a,b,c)$ for any $g\in G$. 
 Thus, by Proposition \ref{howe}, $X_{a,b,c}$ is isomorphic to $X_{a',b',c'}$ over $F$  if and only if  $j_k(a,b,c)=j_k(a',b',c')$ for $k=1,2,3$.
 Observe that $X_{a,b,c}^{\sigma}=X_{\sigma(a),\sigma(b),\sigma(c)}$ is isomorphic to $X_{a,b,c}$ over $F$ if and only if  for $k=1,2,3$ we have
 $$j_k:=j_k(a,b,c)=j_k(\sigma(a),\sigma(b),\sigma(c))=\sigma(j_k(a,b,c)).$$
 Thus $U_{F/K}(X_{a,b,c})=\{\sigma\in {\rm Aut}(F/K): X_{a,b,c}^{\sigma}\cong X_{a,b,c}\}={\rm Aut}(F/K(j_1, j_2, j_3)).$
 Since $L/K$ is a general Galois extension we deduce that
 $$M_{F/K}(X_{a,b,c})={\rm Fix}(U_{F/K}(X_{a,b,c}))=K(j_1,j_2,j_3). $$ \qed
 
\begin{remark}
 Proposition \ref{howe} can be generalized to the case when $F$ does not contain $\qq(i)$.
 In this case $X_{a',b',c'}$ is isomorphic to $X_{a,b,c}$ over $F$ if and only if $g(a,b,c)=(a',b',c')$ for some $g\in \langle g_1,g_2\rangle$ and the field of moduli relative to a general Galois extension $F/K$ equals $K(j_2,j_4,j_5)$ where $j_4(a,b,c)=a+b+c,\ j_5(a,b,c)=a^3+b^3+c^3.$
 \end{remark}



We now consider the Galois extension $\qq(a,b,c)/\qq(j_1,j_2,j_3)$, assuming that $\qq(i)\subset \qq(a,b,c)$.
If $\sigma$ belongs to the Galois group of such extension, then $X_{a,b,c}^{\sigma}\cong X_{a,b,c}$ 
and $\sigma$ acts on $(a,b,c)$ as some $g_{\sigma}\in G$ by Proposition \ref{howe}.
Thus we can define a natural injective group homomorphism
$$\psi: {\rm Aut}(\qq(a,b,c)/\qq(j_1,j_2,j_3))\to G,\ \sigma\mapsto g_{\sigma}.$$
Observe that, if $a,b,c\in \cc$ are generic, then $\psi$ is an isomorphism since the degree of the extension
$\qq(a, b, c)/\qq(j_1, j_2, j_3)$ is $24=|G|$.

\begin{proposition}Let  $a,b,c\in \cc$ such that $a^2,b^2,c^2$ are pairwise distinct, $X_{a,b,c}$ is smooth and $\qq(i)\subset \qq(a,b,c)$. If ${\rm Im}(\psi)\subset \langle g_1,g_2\rangle$, then  $X_{a,b,c}$ can be defined over $\qq(j_1,j_2,j_3)=M_{\qq(a,b,c)/\qq(j_1,j_2,j_3)}(X_{a,b,c})$.
\end{proposition}
\begin{proof}
According to Weil's Theorem \ref{Wa} we need to choose an isomorphism $f_{\sigma}:X_{a,b,c}\to X_{\sigma(a),\sigma(b),\sigma(c)}$ for any $\sigma\in {\rm Aut}(\qq(a,b,c)/\qq(j_1,j_2,j_3))$ such that the following condition holds for all $\sigma,\tau$:
\begin{equation}\label{weil} f_{\sigma \tau}=f_{\tau}^{\sigma}\circ f_{\sigma}.\end{equation}
We assume that ${\rm  Im}(\psi)=\langle g_1,g_2\rangle$, the case when there is just an inclusion is similar.
Let $\sigma_1=\psi^{-1}(g_1)$ and $\sigma_2=\psi^{-1}(g_2)$.
We choose $f_{\sigma_1}(x:y:z)=(x:z:y)$, $f_{\sigma_2}(x:y:z):=(z:x:y)$ and $f_{\sigma}:=f_{\sigma_2}^s\circ f_{\sigma_1}^r$ if $\sigma=\sigma_1^r\circ \sigma_2^s$.
Observe that $f_{\tau}$ is always defined over $\qq$, so that $f_{\tau}^{\sigma}=f_{\tau}$.
Thus condition (\ref{weil}) clearly holds.
\end{proof}
\begin{example}
Consider a plane quartic $X=X_{a,b,c}$ where $a=\alpha, b=\bar \alpha$ with $\alpha\in \qq(i)$ and $c\in \qq$ such that $a^2,b^2,c^2$ are pairwise distinct and the curve is smooth.
By Proposition \ref{fm} the field of moduli of the curve relative to the extension
$\qq\subset\qq(a,b,c)=\qq(i)$ is $\qq$.
The Galois group ${\rm Aut}(\qq(i)/\qq)$ is generated by the complex conjugation $\sigma(z)=\bar z$ and $\psi(\sigma)=g_1$.
An isomorphism between $X$ and $X^\sigma$ is given by  $f_\sigma(x:y:z)=(x:z:y)$.
Since $id=f_{\sigma^2}=f_\sigma^\sigma\circ f_\sigma=(f_\sigma)^2$, then $X$ can be defined over $\qq$.
\end{example}

\bibliographystyle{amsplain}

\end{document}